\documentclass[a4paper,10pt]{amsart}
\usepackage[utf8]{inputenc}
\usepackage{amssymb}
\usepackage{amsmath}
\usepackage{bbm}
\usepackage{mathrsfs}
\usepackage[all]{xy}
\usepackage{manfnt}
\usepackage{pigpen}
\usepackage{bm}
\usepackage{datetime}
\usepackage{geometry}
\usepackage{lscape}
\usepackage{enumerate}

\usepackage[usenames, dvipsnames]{color}
\title{Symmetric powers, Steenrod operations and representation stability}
\author{Geoffrey Powell}
\address{
Laboratoire angevin de recherche en mathématiques (LAREMA),
CNRS, Université d’Angers, 2 Bd Lavoisier 49045 Angers, France.
}
\email{Geoffrey.Powell@math.cnrs.fr}
\urladdr{http://math.univ-angers.fr/~powell/}
\keywords{Functor, strict polynomial functor, polynomial functor, highest weight category, Ringel duality, representation stability, symmetric power, indecomposable, Steenrod algebra, Peterson hit problem.}
\subjclass[2000]{Primary 55S10, 20G43; Secondary 20G05}
\thanks{This work was partially supported by the ANR Project {\em ChroK}, {\tt 
ANR-16-CE40-0003} and by the project {\em Nouvelle Équipe},  convention {\tt 2013-10203/10204} between the Région des Pays de la  Loire and the Université d'Angers.}
\newtheorem{THM}{Theorem}

\newtheorem{COR}[THM]{Corollary}

\newtheorem{thm}{Theorem}[section]
\newtheorem{prop}[thm]{Proposition}

\newtheorem{cor}[thm]{Corollary}
\newtheorem{lem}[thm]{Lemma}
\theoremstyle{definition}
\newtheorem{defn}[thm]{Definition}
\newtheorem{exam}[thm]{Example}

\theoremstyle{remark}
\newtheorem{rem}[thm]{Remark}
\newtheorem{conv}[thm]{Convention}
\newtheorem{nota}[thm]{Notation}

\newtheorem{conj}[thm]{Conjecture}
\newtheorem{CONJ}{Conjecture}
\newcommand{\df}{\overline{\Delta_\kf}}
\newcommand{\per}{\mathrm{Per}}

\newcommand{\sg}[1]{\boxed{#1}}

\newcommand{\degree}[1]{{||#1 ||_p}}
\newcommand{\degt}[1]{{||#1||_2}}

\newcommand{\flt}{\mathrm{Filt}}
\newcommand{\seq}{\mathsf{Seq}}
\newcommand{\prt}{\mathsf{Part}}
\newcommand{\prest}{{\mathsf{Part}^{[p-\mathrm{res}]}}}
\newcommand{\forget}{\mathscr{O}}
\newcommand{\stp}{\mathscr{P}}

\newcommand{\kf}{\mathbbm{k}}
\newcommand{\ft}{{\mathbb{F}_2}}
\newcommand{\fp}{{\mathbb{F}_p}}
\newcommand{\strunc}{\overline{S}}
\renewcommand{\phi}{\varphi}
\renewcommand{\epsilon}{\varepsilon}
\renewcommand{\hom}{\mathrm{Hom}}
\newcommand{\f}{\mathscr{F}}
\newcommand{\cala}{\mathscr{A}}
\newcommand{\calc}{\mathscr{C}}
\newcommand{\zed}{\mathbb{Z}}
\newcommand{\nat}{\mathbb{N}}
\newcommand{\ext}{\mathrm{Ext}}
\newcommand{\sym}{\mathfrak{S}}
\newcommand{\op}{^\mathrm{op}}
\newcommand{\ob}{\mathsf{Ob}\hspace{3pt}}

\newcommand{\eval}{\mathrm{ev}}

\newcommand{\po}{\ar@{}[dr]|(.7){\text{\pigpenfont R}}}
\newcommand{\pb}{\ar@{}[dr]|(.3){\text{\pigpenfont J}}}

\newcommand{\fdvs}{\mathcal{V}^f}
\newcommand{\vs}{\mathcal{V}}

\newcommand{\unst}{\mathscr{U}}

\newcommand{\qa}{\mathfrak{Q}}

\numberwithin{equation}{section}
\begin{document}

\begin{abstract}
Working over the prime field $\fp$, the structure of the indecomposables $Q^*$ 
 for the action of the algebra $\cala (p)$ of Steenrod reduced powers  on the symmetric power functors $S^*$ is studied by exploiting the theory of strict polynomial functors. 
 
 In particular, working at the prime $2$, representation stability is exhibited for certain related functors,  leading to a conjectural representation stability theoretic description of quotients of $Q^*$ arising from the polynomial filtration of symmetric powers. 
\end{abstract}

\maketitle

\section{Introduction}

Symmetric powers give a rich and far from understood source of representations working over a finite field $\kf$: for $V$ a finite-dimensional $\kf$-vector space, $S^n (V)$ is a representation of the general linear group $GL(V)$ and the multiplicities of its composition factors are not known in general. Algebraic topologists have long been interested in these representations, since $S^* (V)$ is the polynomial part of the cohomology of the classifying space $BV^\sharp$ of the dual of $V$ when $\kf$ is the prime field $\fp$.  In particular $\cala (p)$, the mod $p$ algebra of Steenrod reduced powers,  acts upon $S^* (V)$ naturally with respect to $V$ and one can consider the representations given by the indecomposables $Q^* (V) := \fp \otimes_{\cala(p)} S^* (V)$. 

This is of interest since Singer's algebraic transfer relates $Q^* (V)$ to $\ext^{\dim V} _\cala (\fp, \fp)$, the cohomology of the mod $p$ Steenrod algebra $\cala$, and thus to the stable homotopy groups of spheres. To be concrete, in this paragraph take $p=2$, so that $\cala (2)$ identifies with $\cala$;  there is a morphism from the cohomology of the Steenrod algebra to the dual of the $GL(V)$-invariants $Q^* (V) ^{GL(V)}$, with $\dim V$ corresponding to the cohomological grading. To apply this requires understanding the structure of $Q^* (V)$ as a $GL (V)$-module. Much effort has gone into this, which is known as the Peterson hit problem (see the volumes by Walker and Wood \cite{MR3729478,MR3729477}). Few global results  (i.e., results that hold for all $V$) are known and the current research frontier is $\dim V = 5$ (see the work of Nguyen Sum  \cite{MR3409252,MR3318156} for example).

The approach taken  here is  to consider $Q^n$ as an object of $\f$, the category of functors on finite-dimensional $\fp$-vector spaces (see Section \ref{sect:functors} for background). This is of Eilenberg-MacLane polynomial degree $n$  with cosocle  the $n$th $p$-truncated symmetric power $\strunc^n$, which identifies with the $n$th exterior power $\Lambda^n$ when $p=2$.

In order to get further general information on  $Q^n$, the  filtration 
\[
\ldots \subset Q^n [d] \subset Q^n [d+1] \subset \ldots \subset Q^n
\]
induced by  the polynomial filtration  $(p_d S^n)_{d \in \nat} \subset S^n$ is considered, in particular, 
the associated subquotients $Q^n_d := Q^n [d]/ Q^n [d-1]$. By construction, $Q^n_d$ is a polynomial functor of degree $d$ and is zero if $d>n$; for $d=n$ one recovers the cosocle of $Q^n$, namely $\strunc^n$.

An important observation is that these subquotients can be approximated by using the action of  $\cala(p)$ on the associated graded $p_d S^*/p_{d-1} S^*$ of the polynomial filtration, by forming the indecomposables $\qa^*_d :=\fp \otimes_{\cala (p)} p_d S^*/p_{d-1} S^*$. There is a natural surjection:
\[
\qa^n_d \twoheadrightarrow Q^n_d
\] 
that is close to being an isomorphism; Corollary \ref{cor:refine_first_occurrence} gives a criterion ensuring it is an isomorphism, whereas Corollary \ref{cor:qa_different_Q} gives an  example where it is not. The  problem of understanding the kernel of this map is of significant interest, especially in the context of the results of this paper and the conjectures formulated below.  

A crucial fact is that $p_d S^n/p_{d-1} S^n$ has the structure of a  strict polynomial functor of degree $d$ (see Section \ref{sect:strict} for more details). The category $\stp_d$ of degree $d$ strict polynomial functors  comes equipped with a forgetful functor $\forget : \stp_d \rightarrow \f_d$ to Eilenberg-MacLane polynomial functors of degree $d$, but the structure of  $\stp_d$ is much more rigid than that of  $\f_d$; for instance $\stp_d$ is a highest weight category. This can be exploited in studying $\qa^n_d$ and hence the functors $Q^n$.

 More explicitly, there is an isomorphism of functors
\begin{eqnarray}
\label{eqn:filt_quot}
p_d S^n/p_{d-1} S^n
\cong 
\bigoplus _{\omega \in \seq^p_d (n) } 
\bigotimes_{i \in \nat}\strunc^{\omega_i},
\end{eqnarray}
where the sum is taken over  sequences  $\omega$ of natural numbers  such that $\sum \omega_i = d$ and $\sum \omega_i p^i =n$. Here, not only does the right hand side arise from the category $\stp_d$ of strict polynomial functors but so does the $\cala (p)$-action  upon $p_d S^*/p_{d-1} S^*$, so that $\qa^n_d$ has the structure of a strict polynomial functor.

It is instructive to consider the functors $\qa^n_n = \strunc^n$ as forming a periodic family in $n$ of period one. This is especially transparent when the prime is $2$, in which case $\strunc^n \cong \Lambda^n$; likewise, the functors $\qa^n_{n-1}$ (for $n \geq 3$) form  periodic families of period $2$ given by dual Weyl functors (see Example \ref{exam:qa_small}). This can be explained and generalized  by representation stability, as outlined below.  

Write $\stp_d^{\geq c}$, $c \in \nat$, for the full subcategory of strict polynomial functors with underlying functor vanishing on $\fp^{c-1}$. Harman \cite{2015arXiv150906414H} proved that these categories are highest weight categories and that they satisfy a form of  representation  stability, as reviewed here  in Theorem \ref{thm:representation_stability}. Namely, for natural numbers $e > d, t$
 such that  $d \geq 2t$ and  $e \equiv d \mod p^{\lceil \log_p t \rceil}$, there is an equivalence of categories:
\[
\per_{d,e} : \stp_d^{\geq d-t} \stackrel{\cong}{\rightarrow} \stp_e ^{\geq e-t}.
\]
The simple objects of $\stp_d$  are denoted $L_\lambda$, indexed by partitions $\lambda$ of $d$,  and the periodicity equivalence acts on such simple objects via $\per_{d,e} (L_\lambda) = L_{\lambda  \bullet 1^{e-d}}$, where  $\lambda  \bullet 1^{e-d}$ is the partition of $e$ given by concatenation.

For the application of this result, the prime $p$ has to be taken to be two. This is due to the distinguished rôle played by tensor products of exterior powers in the higher weight structures and the fact that,  for $p=2$, the associated graded of the  polynomial filtration of the symmetric powers is also given by such functors,  by  equation (\ref{eqn:filt_quot}), since $\strunc^n \cong \Lambda^n$. 

\begin{THM}
(Theorem \ref{thm:qa_rep_stab}.)
Let $\kf = \ft$. For natural numbers $d \leq n$ such that $d \geq 2( n-d) $, if $e > d \in \nat$ such that $e \equiv d \mod 2^{\lceil \log_2 (n-d) \rceil}$  then $\qa^n _d \in \ob \stp_d^{\geq d- (n-d)}$ and 
$\qa^n _e \in \ob \stp_e^{\geq e- (n-d)}$ and,  under the equivalence of categories $\per_{d,e} 
 :
\stp_d^{\geq d- (n-d)}
\stackrel{\cong}{\longrightarrow}
 \stp_{e}^{\geq e - (n-d)} $ 
 of Theorem \ref{thm:representation_stability}, $\qa^n_d$ is sent to 
$
\qa_{e}^{n + e - d }$.
\end{THM}

Here, by convention, $\lceil \log_2 0 \rceil =0$;  hence, for $n=d$, the above corresponds to the fact that the categories $\stp_d ^{\geq d}$, for $d \in \nat$, are all equivalent to the category of $\ft$-vector spaces and the Theorem  explains the periodicity of $\qa^d_d$ alluded to above.

The next step is to pass to $Q^n_d$. Here there are two difficulties: a lack of understanding of the kernel of $\qa^n _d \twoheadrightarrow Q^n_d$ in general, together with the fact that one can no longer work entirely with strict polynomial functors.

The analogous functor categories $\f_d^{\geq d - (n-d)}$ are not known to satisfy representation stability, although weaker results can be established (these are not developed here). One can, however, work at the level of the Grothendieck groups of the relevant categories, for which weak representation stability can be proved:

\begin{THM}
\label{THM:1}
(Theorem \ref{thm:rep_stab_F}.)
Suppose that $\kf = \fp$ and that  $d,t \in \nat$ such that $d>2t$.  For  $d<e \in \nat$ such that $d \equiv e \mod p^{\lceil \log_p t \rceil}$,  the periodicity equivalence $\per_{d,e} : \stp_d^{\geq d-t} \rightarrow \stp_e^{\geq e-t}$ induces a commutative diagram of abelian groups:
\[
\xymatrix{
G_0 (\stp_d^{\geq d -t})
\ar[d]_{G_0(\per_{d,e})}^\cong
\ar[r]^{G_0 (\forget)} 
&
G_0 (\f_d^{\geq d -t})
\ar[d]^{\bullet 1^{e-d}}_\cong
\\
G_0 (\stp_e^{\geq e-t})
\ar[r]_{G_0 (\forget)} 
&
G_0 (\f_e^{\geq e -t})
}
\]
in which the vertical morphisms are isomorphisms, where $\bullet 1^{e-d}$ is induced by concatenation of partitions.
\end{THM}

With this in hand, the following is immediate:

\begin{COR}
\label{COR:1}
(Corollary \ref{cor:qa_rep_stab_F}.)
Let $\kf = \ft$,  $d, e, n  \in \nat$ satisfy the hypotheses of Theorem \ref{THM:1} and suppose in addition that $d > 2 (n-d)$. Then, under the isomorphism
$
\bullet 1^{e-d} : 
G_0 (\f_d^{\geq d - (n-d)})
\stackrel{\cong}{\rightarrow} 
G_0 (\f_d^{\geq e - (n-d)})$ of Grothendieck groups, $ [\qa^n_d]$ maps to 
$
[\qa_e^{n+e-d}]$.
\end{COR}

However, rather than the subquotients $\qa^n_d$ (and hence $Q^n_d$), it is the quotients of the form $Q^n / Q^n [d-1]$ that are of interest. Corollary \ref{COR:1} suggests:

\begin{CONJ}
\label{CONJ:1}
(Conjecture \ref{conj:rep_stab_Q}.)
Let $\kf = \ft$ and suppose that $d, e, n \in \nat$ satisfy the hypotheses of Corollary \ref{COR:1}, then under the isomorphism $\bullet 1^{e-d} : 
G_0 (\f_n^{\geq d - (n-d)})
\stackrel{\cong}{\rightarrow} 
G_0 (\f_{n+e-d}^{\geq e - (n-d)})$
 of Grothendieck groups, 
$  \big[Q^n/ Q^n[d-1] \big]$ maps to
$\big[Q^{n+ e-d} /Q^{n+ e-d} [e-1]\big ]$.
\end{CONJ}

Initial calculations suggest that a stronger result should be true: 

\begin{CONJ}
\label{CONJ:2}
(Conjecture \ref{conj:rep_stab_Q_strong}.)
Let $\kf = \ft$ and suppose that $d, e, n \in \nat$ satisfy the hypotheses of Corollary \ref{COR:1}, then 
the lattices of subobjects of $Q^n /Q^n[d-1]$ and $Q^{n+e-d}/ Q^{n+e-d}[e-1]$ are isomorphic,  compatibly with the identification of Conjecture \ref{CONJ:1}.
\end{CONJ}

These conjectures provide a new approach to understanding the structure of the indecomposables $Q^*$ and hence of the symmetric power functors $S^*$. 
For instance, given $n \in \nat$, rather than simply studying the structure of $Q^n$, one should first determine the largest $d$ for which $Q^n[d-1]=0$; this is possible using methods developed by Wood and others (cf. \cite{MR3729477}), essentially exploiting instability of the $\cala$-module structures.
 Then one considers the family of functors $Q^{n+N} / Q^{n+N}[d+N-1]$, where $N \equiv 0 \mod 2^{\lceil \log_2 (n-d) \rceil]}$. For large $N$, the conjectures assert that this family exhibits a form of representation stability, so that it is a reasonable strategy to commence by determining its structure, before considering the low dimensional `noise'.

This conjectural periodicity cannot be proved by considering $Q^* (V)$ for a fixed finite-dimensional $V$, since the quotients $  
Q^{n+e-d}/ Q^{n+e-d}[e-1]$  vanish on $V$ for $e \gg 0$. A contrario, this viewpoint will shed light on the structure of $Q^* (V)$; in particular giving new ways of analysing the structure.

Similar properties should be exhibited at odd primes; however, it will be necessary to develop the appropriate framework for studying the representation stability.

\bigskip
{\bf Organization of the paper:}
Background on functors and on strict polynomial functors is provided in Sections \ref{sect:functors} and  \ref{sect:strict}, and Section \ref{sect:hwc} reviews highest weight categories.

Section \ref{sect:symm_poly} introduces the polynomial filtration of symmetric powers and  Section \ref{subsect:nat_trans} reviews the natural transformations between symmetric power functors, thus establishing the link with Steenrod operations. The indecomposables $Q$ are then introduced in Section \ref{sect:Q}, together with their approximations $\qa$. 

The representation stability results for strict polynomial functors are recalled    in Section \ref{sect:stp_rp} and Section \ref{sect:weak} proves a weak version of representation stability for polynomial functors. Finally, Section \ref{sect:rep_stab} puts everything together at the prime $2$, stating the main results and the Conjectures.

\begin{nota}
Throughout, 
$\nat$ denotes the non-negative integers and $\fp$  the prime field of characteristic $p>0$.
\end{nota}

\section{Functors on $\kf$-vector spaces}
\label{sect:functors}

This section reviews the theory of functors between $\kf$-vector spaces; in the applications, $\kf$ will be taken to be the prime field $\fp$. This material is readily available in standard references such as \cite{MR1269607},  \cite{MR1726705}.

\subsection{The category of functors}

Let  $\kf$ be a field,   $\vs$ be the category of $\kf$-vector spaces and $\fdvs \subset \vs$ the full subcategory of finite-dimensional vector spaces. The dual of a vector space $V$ is denoted $V^\sharp$.

\begin{nota}
\label{nota:functors}
Let $\f$ denote the category of functors from $\fdvs$ to $\vs$.  The duality functor $D : \f\op \rightarrow \f$ is  defined by $D F(V):= F(V^\sharp) ^\sharp$.
\end{nota}

The category $\f$ inherits an abelian structure  from $\vs$ and has tensor product $\otimes$ defined pointwise, with unit the constant functor $\kf$. A functor is said to be finite if it has a finite composition series.

\begin{exam}
\label{exam:functors}
For $d \in \nat$, the $d$th tensor power $T^d $ is given by $V \mapsto V ^{\otimes d}$. The symmetric group $\sym_d$ acts by place permutations on $T^d$ and the $d$th divided power $\Gamma^d$ is given by the invariants $(T^d)^{\sym_d}$ and the $d$th symmetric power  $S^d$ by the coinvariants $(T^d)_{\sym_d}$. The $d$th exterior power  is denoted by $\Lambda^d$; this is both a subfunctor and a quotient of $T^d$.

These functors are all finite;  $T^d$, $\Lambda^d$  are self-dual under $D$, whereas $D \Gamma^d \cong S^d$. 
\end{exam}

\begin{exam}
If $\kf$ is a field of characteristic $p>0$,
the cosocle of $S^d$ is the $d$th $p$-truncated symmetric power $\strunc^d$ which has presentation:
\[
(S^1)^{(1)}\otimes S^{d-p} 
\rightarrow 
S^d 
\rightarrow 
\strunc^d
\rightarrow 
0  
\]
where the left hand map is given by $x \otimes y \mapsto x^p y$, using the product structure of symmetric powers; here the $^{(1)}$ denotes the Frobenius twist that ensures $\kf$-linearity and, if $d<p$, $S^{d-p}$ is understood to be  zero. The functor $\strunc^d$ is self-dual and it identifies with $\Lambda^d$ if $p=2$. 
\end{exam}

\subsection{Exponential functors}
\label{subsect:expo}

Exponential functors provide a powerful calculational tool. As a general reference, the reader is referred to \cite{MR1726705}; note  that the exponential functors considered here are the Hopf exponential functors of {\em loc. cit.}.

Consider the category $\fdvs$ as a symmetric monoidal category with respect to $\oplus$.

\begin{defn}
\label{defn:expo} 
For $(\calc, \otimes, \mathbbm{1})$ a symmetric monoidal category, the category of exponential functors from $\fdvs$ to $\calc$ is the category of strict monoidal functors from $\fdvs$ to $\calc$. Thus, a functor $E$ is exponential if, for $V, W \in \ob \fdvs$, there is a natural isomorphism $E (V \oplus W) \cong E (V) \otimes E (W)$, and these satisfy the associativity, symmetry  and unit axioms. 
\end{defn}

We work here with  exponential functors taking values in  $\vs^\nat$, the category of $\nat$-graded vector spaces, equipped with the usual graded tensor product (and the symmetry without Koszul signs) and refer to these as graded exponential. 

\begin{rem}
The additive structure of $\fdvs$ induces additional structure on exponential functors. For example, exponential functors to $\vs$ takes values  in bicommutative Hopf algebras.
\end{rem}

\begin{exam}
\label{exam:expo}
Let $\kf$ be an arbitrary field. 
\begin{enumerate}
\item 
The functor $S^*$ is graded exponential. In particular there are natural commutative products $S^i \otimes S^j \rightarrow S^n$ and cocommutative coproducts $S^n \rightarrow S^i \otimes S^j$, where $n= i+j$. 
\item 
If $\kf$ has characteristic $p>0$, $\strunc^*$ is  graded exponential and the natural surjection $S^* \rightarrow \strunc^*$ is a morphism of graded exponential functors.
\end{enumerate}
\end{exam}

Further examples are obtained by forming tensor products:

\begin{lem}
\label{lem:property_expo}
Let $E, E'$ be graded exponential functors then 
the graded tensor product $E\otimes E'$ is graded exponential. 
\end{lem}

The following shows the usefulness of the exponential property:

\begin{prop}
\label{prop:hom_expo_tensor_product}
For $E$  a graded exponential functor and $F, G \in \ob \f$, there are natural graded isomorphisms:
\begin{eqnarray*}
\hom _\f (E,F \otimes G) &\cong & \hom_\f (E , F) \otimes \hom_\f (E , G) 
\\
\hom _\f (F \otimes G, E) &\cong & \hom_\f (F, E) \otimes \hom_\f (G, E).
\end{eqnarray*}
\end{prop}

\begin{cor}
\label{cor:expo_hom_bicommutative_Hopf_alg}
For $E$ a graded exponential functor, the bigraded object $\hom_\f (E , E)$ has  a natural bigraded, bicommutative Hopf algebra structure.
\end{cor}

\subsection{Polynomial functors}

The difference functor $\df : \f \rightarrow \f$ recalled below leads to the following simple definition of polynomial functors in the sense of Eilenberg and MacLane. 

\begin{nota}
\label{nota:diff}
Denote by 
\begin{enumerate}
\item
$\Delta_\kf : \f \rightarrow \f$  the shift functor given by $ \Delta _\kf F (V) := F (V \oplus \kf)$; 
\item 
$\df : \f \rightarrow \f$ the difference functor given by $\df F (V):= \Delta _\kf F (V) / F(V)$ for the canonical inclusion $F(V) \hookrightarrow F(V \oplus \kf)$.
\end{enumerate}
\end{nota}

\begin{defn}
\label{defn:poly}(Cf. \cite{MR1269607}.)
A functor $F \in \ob \f$ is (Eilenberg-MacLane) polynomial of degree $\leq d \in \nat$ if  $\df^{d+1} F =0$. 

The full subcategory of functors of polynomial degree at most $d$ is denoted $\f_d \subset \f$ and the right adjoint to the inclusion by  $p_d : \f \rightarrow \f_d$. 
\end{defn}

\begin{exam}
\label{exam:functors_poly_d}
The functors of Example \ref{exam:functors} are  polynomial of degree $d$. 
\end{exam}

\subsection{Stratification by rank}

\begin{defn}
\label{defn:stratification_f}
For $c, d \in \nat$, let $\f^{\geq c} \subset \f$ be the kernel of the evaluation functor 
$\eval_{c-1} : F \mapsto F (\kf^{c-1})$ and  $\f_d ^{\geq c}$ denote the intersection of the full subcategories $\f_d$ and $\f^{\geq c}$ in $\f$. 
\end{defn}

\begin{exam}
\ 
\begin{enumerate}
\item 
The $d$th exterior power $\Lambda^d$ belongs to 
$\f^{\geq d} \backslash \f^{\geq d+1}$. 
\item 
The $d$th symmetric power $S^d$ belongs to $\f^{\geq 1} \backslash \f^{\geq 2}$ if $d>0$. 
\end{enumerate}
\end{exam}

There is a filtration:
\[
\ldots \subsetneq \f^{\geq c+1} \subsetneq \f^{\geq c} \subsetneq \ldots \subsetneq \f^{\geq 0} = \f
\]
(cf. \cite[Remark 2.9]{MR1300547}) and an analogous stratification of $\f_d$ by the subcategories $\f^{\geq c}_d$. The latter has  finite length by the following:

\begin{prop}
\label{prop:finite-stratification}
For $c,d \in \nat$, the difference functor restricts to $\df :  \f^{\geq c}_d \rightarrow \f^{\geq c-1}_{d-1}$ that is faithful if $c>0$. In particular, if $c > d \in \nat$, then $\f_d ^{\geq c}$ is $0$.
\end{prop}

\begin{proof}
The key point is  that, if $F \in \f^{\geq 1}$, then $F=0$ if and only if $\df F=0$. 
\end{proof}

\section{Strict polynomial functors}
\label{sect:strict}

This section reviews the basic theory of strict polynomial functors, working over a field $\kf$.
The exposition is based largely  upon that of Krause \cite{MR3077659,MR3644799} (where the more general case of $\kf$ a commutative ring is considered) and of Kuhn (cf. \cite{MR1918209} for example); see also Friedlander et al. \cite{MR1726705}.

\subsection{Basic structure}

\begin{defn}
\label{defn:stp}
For $d \in \nat$, let
\begin{enumerate}
\item 
 $\Gamma^d \fdvs$ be the $\kf$-linear category with objects $V \in \ob \fdvs$ and morphisms $\hom_{\Gamma^d \fdvs} (V, W) := \Gamma ^d (\hom_{\fdvs} (V, W))$;
\item 
$\stp_d$, the  category of degree $d$ strict polynomial functors,  be  the category of $\kf$-linear functors from $\Gamma ^d \fdvs$ to $\vs$.  
\end{enumerate} 
The category $\stp$ of strict polynomial functors is $\bigoplus_{d \in \nat} \stp_d$. 
\end{defn}

\begin{prop}
\cite{MR1726705}
\label{prop:stp}
For $d \in \nat$, the category $\stp_d$ is abelian with enough projectives and enough injectives.  There is an exact, faithful forgetful functor $\forget : \stp_d \rightarrow \f$ that takes values in $\f_d$. 
\end{prop}

There is important additional structure (for $d, e \in \nat$):
\begin{enumerate}
\item 
The (external) tensor product:
$ 
\otimes : \stp_d \times \stp_e \rightarrow \stp_{d+e}
$.
  \item 
Duality $D : \stp_d \op \rightarrow \stp_d$.
\item 
The Frobenius twist for $\kf$ a field of characteristic $p$. For $r \in \nat$, the $r$th iterated Frobenius  defines $I^{(r)} \in \ob \stp_{p^r}$ and the Frobenius twist functor  
$
(-)^{(r)} : \stp_d \rightarrow \stp _{d p^r}
$  
is  given by precomposition  $- \circ I^{(r)}$.  
\end{enumerate}
These structures are compatible with their counterparts for $\f$ via the forgetful functor $\forget$.

\begin{rem}
The behaviour of the Frobenius twist for $\kf$ a finite field is fundamentally different for strict polynomial functors as opposed to $\f$, since the Frobenius twist is an equivalence of categories on $\f$ when $\kf$ is finite. More particularly, over the prime field $\fp$, the Frobenius twist is the identity on $\f$; this fact will be exploited below. 
\end{rem}

\begin{exam}
\label{exam:strict_functors}
For $d \in \nat$, $T^d$, $S^d$, $\Gamma^d$ and $\Lambda^d$ are canonically strict polynomial  of weight $d$, as is $\strunc^d$ when $\kf$ has characteristic $p$.  
\end{exam}

\begin{rem}
\label{rem:projective_gen_schur}
For  $n \geq d \in \nat$, $\Gamma ^d \circ \hom (\kf^n , -)$ is a projective generator of $\stp_d$. This provides the link with the classical Schur algebras: the Schur algebra $S (n,d)$ is  the endomorphism ring $\mathrm{End}_{\stp_d} (\Gamma ^d \circ \hom (\kf^n , -))$ and,  for $n \geq d$,  $\stp_d$ is equivalent to the category of $S (n,d)$-modules.
\end{rem}

\subsection{Exponential strict polynomial functors}

The theory of exponential functors (see Section \ref{subsect:expo}) also applies in the context of strict polynomial functors, as in  \cite{MR1726705}.

\begin{rem}
\label{rem:expo_stp}
 Using the definition of the categories $\stp_d$ given here, one uses the fact that $(\fdvs, \oplus)$ induces a (graded) symmetric monoidal structure on the categories $\Gamma^* \fdvs$, so that $
\oplus : \Gamma^d \fdvs \times \Gamma ^e \fdvs \rightarrow \Gamma^{d+e}\fdvs
$
 for $d, e \in \nat$.
 \end{rem}

The main properties and applications of exponential functors carry over, {\em mutatis mutandis}. For example, this leads to:

\begin{prop}
\label{prop:expo_compare_forget}
Let $\kf$ be a field of characteristic $p>0$. For $d \in \nat$ and $\omega := \{\omega_i |i \in \mathscr{I} \}$, $\eta := \{ \eta_j | j \in \mathscr{J} \}$  sequences of natural numbers such that $\sum_{i \in \mathscr{I}} \omega_i  = \sum_{j \in \mathscr{J} } \eta_j =d$, the forgetful functor $\forget : \stp_d \rightarrow \f$ induces an isomorphism:
\[
\forget : \hom_{\stp_d} (\strunc^\omega , \strunc^\eta) 
\stackrel{\cong}{\rightarrow}
\hom_{\f} (\strunc^\omega , \strunc^\eta),
\]
where $\strunc ^\omega := \bigotimes _{i \in \mathscr{I}} \strunc^{\omega_i} $ and $\strunc ^\eta := \bigotimes _{ j\in \mathscr{J}} \strunc^{\eta_j} $.
\end{prop}

\begin{proof}
The forgetful functor $\forget$ induces a monomorphism of $\kf$-vector spaces, hence it suffices to show that they are finite-dimensional of the same dimension. 

Using the exponentiality of $\strunc^*$, this reduces to the fact that $\forget$ induces an isomorphism
\[
\hom_{\stp} (\strunc^m , \strunc^n) 
\cong 
\hom_{\f} (\strunc^m , \strunc^n) 
\cong 
\left\{
\begin{array}{ll}
\kf & m=n\\
0 & \mathrm{otherwise},
\end{array}
\right.
\]
which follows from the simplicity of the functors $\strunc^n$ in the respective categories, with endomorphism ring $\kf$.
\end{proof}

\subsection{Stratifying $\stp_d$}
The stratification of $\f$ by the categories $\f^{\geq c}$ induces a stratification of $\stp$:

\begin{defn}
\label{defn:strat_stp}
For $c,d \in\nat$, let  $\stp_d^{\geq c} \subset \stp_d$ be the full  
subcategory with objects  $P$  such that $\forget P(\kf^{c-1})=0$.
\end{defn}

\begin{lem}
\label{lem:strat_stp}
For $c, d \in \nat$,
\begin{enumerate}
\item 
$\stp_d ^{\geq c} =0$ if $c>d$; 
\item 
$\stp_d^{\geq 0} = \stp_d$; 
\item 
$\stp_d^{\geq d}$ is equivalent to the category of $\kf$-vector spaces.
\end{enumerate}
\end{lem}

This gives a filtration
\[
0 \subsetneq \stp_d^{\geq d} \subsetneq \stp_d^{\geq d-1} \subsetneq  \ldots \subsetneq \stp_d^{\geq 2} \subsetneq \stp_d^{\geq 0} = \stp_d .
\]

The tensor product behaves well with respect to these subcategories,  in particular:

\begin{prop}
\label{prop:tensor_compatib}
For $c, d,e \in \nat$, the  tensor product restricts to 
$
\otimes : 
\stp_d^{\geq c} \times \stp_e 
\rightarrow 
\stp_{d+e} ^{\geq c} .
$
If $e < c$ then the following diagram is a pullback of categories:
\[
\xymatrix{
\stp^{\geq c}_d \times \stp_e 
\pb
\ar[r]
\ar@{^(->}[d]
&
\stp^{\geq c}_{d+e}
\ar@{^(->}[d]
\\
\stp_d \times \stp_e 
\ar[r]_\otimes 
&
\stp_{d+e}.
}
\]
\end{prop}

\section{Simple objects and the highest weight  structure}
\label{sect:hwc}

The category $\stp_d$ of strict polynomial functors of
 degree $d$ forms a highest weight category. Such categories were introduced by Cline, Parshall and Scott (see \cite{MR961165} for instance) and have important applications. In particular, the structure of $\stp_d$ is more rigid than that of the category $\f_d$ of functors of Eilenberg-MacLane polynomial degree $d$.

\subsection{Weyl functors, simples and the highest weight structure}

Partitions of $d$ index the simple objects of $\stp_d$ and arise in describing the highest weight structure of the category. It is convenient here to index  sequences of non-negative integers by $\nat$ rather than positive integers.

\begin{nota}
\label{nota:seq_part}
 \ 
\begin{enumerate}
\item 
For $\lambda \in \nat ^\nat$, 
$|\lambda | := \sum_{i \in \nat} \lambda_i$.
\item 
Let $\seq \subset \nat^\nat$ denote the subset of sequences $\lambda$ such that $|\lambda|< \infty$  and $\prt \subset \seq$ the subset of partitions, namely $\lambda$ such that $\lambda_i \geq \lambda_{i+1}$ for all $i \in \nat$.  
\item 
For $\lambda \in \prt$, the length $l (\lambda)$  of $\lambda$ is zero if $\lambda =0$, otherwise 
$l (\lambda) = 1 + \sup \{ i |\lambda_i \neq 0 \}$.
\item 
For $\lambda \in \prt$, $\lambda' \in \prt$ denotes the conjugate partition.
\item 
For $d \in \nat$, let $\seq_d\subset \seq$ denote the set of $\lambda$ such that $|\lambda |=d$; similarly  $\prt_d := \seq_d \cap \prt$ is the set of partitions of $d$.  
\end{enumerate}
\end{nota}

\begin{nota}
\label{nota:F^lambda}
For $\{ F^n \in \ob \stp_n |n \in \nat \}$ an $\nat$-graded strict polynomial functor with $F^0=\kf \in \stp_0$, and $\lambda \in \seq$, set 
$
F^\lambda := \bigotimes _{i\in \nat} F^{\lambda_i},
$
so that $F^\lambda \in \ob \stp_{|\lambda|}$.   
\end{nota}

\begin{nota}
\label{nota:Weyl_functors}
For $\lambda \in \prt_d$, let $W_\lambda \in \ob \stp_d$ denote the associated Weyl functor  that embeds canonically in $\Lambda^{\lambda'}$ (see \cite[Section 2.3]{MR3644799}).
\end{nota}

The simple objects of $\stp_d$ are described as follows:

\begin{prop}
\label{prop:simple_stp}
For $d \in \nat$ and $\lambda \in \prt_d$, $W_\lambda$ has simple cosocle denoted $L_\lambda$ and $\{ L_\lambda | \lambda \in \prt_d \}$ is a set of representatives of the isomorphism classes of simple objects of $\stp_d$. 

Moreover, for $\lambda \in \prt_d$,
\begin{enumerate}
\item  
 $L_\lambda \in \ob \stp_{|\lambda |} ^{\geq \lambda_0'}$ and $\forget L_\lambda (\kf^{\lambda_0'} ) \neq 0$;
 \item 
$L_\lambda$ is self-dual: i.e., $D L_\lambda \cong L_\lambda$.
\end{enumerate}
\end{prop}

\begin{exam}
\label{exam:simples}
For $d\in \nat$, 
\begin{enumerate}
\item 
$L_{(1^d)} \cong \Lambda^d$. 
\item 
For $p$ a prime,  $\strunc^d \cong L_{((p-1)^a , b)}$ where  $d= a (p-1) + b$, with $0 \leq b < p-1$.
\end{enumerate}
\end{exam}

\begin{defn}
\label{defn:poset_dominance}
For $d\in \nat$, the dominance order $\trianglelefteq$ on $\prt_d$ defined by $\mu \trianglelefteq \lambda$ if and only if $\sum_{i=0}^t \mu_i \leq \sum_{i=0}^t \lambda_i$ for all $t \in \nat$.  
\end{defn}

For $\kf$ a field, the following result is due to Donkin (cf. \cite{MR1200163}). 

\begin{thm}
\label{thm:hwc_stp}
For $d \in \nat$, the category $\stp_d$ is a highest weight category with weights $(\prt_d, \trianglelefteq)$ and standard objects the Weyl functors $ W_\lambda$.  
\end{thm}

\begin{defn}
\label{defn:filt_Delta}
For $d \in \nat$, 
\begin{enumerate}
\item 
the category $\flt_d (\Delta)$ of $\Delta$-good objects is the full subcategory of $\stp_d$ of objects that admit a finite filtration with filtration quotients  of the form $W_\mu$ for $\mu \in \prt_d$;
\item 
dually, the category $\flt_d (\nabla)$ of $\nabla$-good objects is the full subcategory of $\stp_d$ of objects that admit a finite filtration with filtration quotients  of the form $DW_\mu$ for $\mu \in \prt_d$.
\end{enumerate}
\end{defn}

\begin{exam}
For $d \in \nat$ and $\lambda \in \prt_d$, the projective $\Gamma^\lambda$ belongs to $\flt_d (\Delta)$; this corresponds to part of the highest weight structure of $\stp_d$.
\end{exam}

The category $\flt_d (\Delta) \cap \flt_d (\nabla)$ plays an important rôle in tilting theory for highest weight categories and in  Ringel duality theory \cite{MR1128706}, hence the following complement to Theorem \ref{thm:hwc_stp} is important:

\begin{prop}
\label{prop:characteristic}
For $d \in \nat$ and $\lambda\in \prt_d$, $\Lambda^{\lambda'}$ is an indecomposable object of $\flt_d (\Delta) \cap \flt_d (\nabla)$.
\end{prop} 

\begin{rem}
\label{rem:characteristic}
Proposition \ref{prop:characteristic} exhibits $\Lambda^{\lambda'}$  as the {\em minimal tilting object} or {\em characteristic object} associated to $\lambda$ (see \cite{MR1128706,MR3644799}). 
\end{rem}

\subsection{The Steinberg tensor product theorem}

In this section, $\kf$ is a field of characteristic $p>0$. Recall the following standard definition:

\begin{defn}
\label{defn:p-regular}
For $p$ a prime and $\lambda$ a partition, 
\begin{enumerate}
\item 
$\lambda$ is $p$-regular if $\lambda_i >\lambda_{i+p-1}$ for all $i\in \nat$ such that $\lambda_i > 0$;
\item 
$\lambda$ is $p$-restricted if the conjugate partition  $\lambda'$ is $p$-regular (equivalently, $\lambda_i - \lambda _{i+1} <p$ for all $i \in \nat$). 
\end{enumerate}
The set of $p$-restricted partitions is denoted $\prest$. 
\end{defn}

\begin{lem}
\label{lem:partitions_p-restrict}
Let $p$ be a prime. For $\lambda \in \prt$, there is a unique set of $p$-restricted partitions $\lambda[i]\in \prest$, $i \in \nat$,  such that $\lambda = \sum_{i \in \nat} p^i \lambda [i]$, where the sum and scalar multiplication is formed termwise.
\end{lem} 

\begin{proof}
The $p$-restricted partitions $\lambda[i]$ are determined by 
$
\lambda_j - \lambda_{j+1}
= 
\sum_i (\lambda [i] _j - \lambda[i]_{j+1}) p^i,
$ 
for $i, j \in \nat$, where the right hand side corresponds to the $p$-adic expansion of $\lambda_j - \lambda_{j+1}$. 
\end{proof}

Kuhn \cite{MR1918209} proved a  Steinberg tensor product theorem for the category $\f$ of functors over a finite field $\kf$; this has an analogue for   strict polynomial functors which shows how the simple functors $L_\lambda$ indexed by the $p$-restricted partitions generate all simple functors via the Frobenius twist and the tensor product.

\begin{thm}
\label{thm:Steinberg_tensor_prod}
For $\kf$  a field of characteristic $p>0$ and a partition $\lambda$, there is an isomorphism 
$ 
L_\lambda 
\cong 
\bigotimes _{i \in \nat} L_{\lambda [i]}^{(i)}
$  
in $\stp_{|\lambda|}$. In particular, 
$
L_\lambda 
\cong 
 L_{\lambda [0]} 
 \otimes 
 L_{\overline{\lambda}}^{(1)}
$
in $\stp_{|\lambda|}$,  
for a partition $\overline{\lambda}$, where $|\lambda | = |\lambda [0]| + p |\overline{\lambda}|$.
\end{thm}

\begin{cor}
\label{cor:strict_poly_vs_EM}
Let $\kf$ be a field of characteristic $p>0$ and $\lambda \in \prt_d$ be a partition. Then $\forget (L_\lambda)  \in \ob  \f $ has polynomial degree $\sum_i |\lambda [i]|$; this is equal to $d$ if and only if $\lambda = \lambda [0]$ (i.e., $\lambda$ is $p$-restricted). 
\end{cor}

The above results are related to the classification of the simple objects of $\f$ (cf. \cite{MR1300547}):

\begin{prop}
\label{prop:classify_simples_fp}
Let $\kf$ be the prime field $\fp$. Then the set 
$\{ \forget (L_\lambda) | \lambda \in \prest \}$ represents the set of  isomorphism classes of simple objects of $\f$. In particular, the Grothendieck group $G_0 (\f)$ is isomorphic to the free abelian group on $\prest$.
\end{prop}

\section{Symmetric powers and the polynomial filtration}
\label{sect:symm_poly}

This section explains how the polynomial filtration of symmetric powers can be studied by using the theory of strict polynomial functors.

\subsection{Filtering $S^r$ in $\stp_r$ and the polynomial filtration}

Throughout this subsection, $\kf$ is a field of characteristic $p$.

\begin{defn}
\ 
\begin{enumerate}
\item 
For a sequence $\omega  \in \seq $, define  $\degree{\omega} := \sum_{i \in \nat} \omega_i p^i $.
\item 
For $r \in \nat$, let $\seq^p(r)\subset \seq$ denote the set of sequences $\omega$ such that $\degree{\omega}=r $ and set $\seq_d^p (r) := \seq_d \cap \seq^p(r)$. 
\end{enumerate}
\end{defn}

\begin{lem}
For  $r \in \nat$, $\seq^p (r)$ is a finite set.
\end{lem}

The following is analogous to $F^\lambda$ introduced in Notation \ref{nota:F^lambda}:

\begin{nota}
For $\{ F^n \in \ob \stp_n |n \in \nat \}$ an $\nat$-graded strict polynomial functor with $F^0=\kf$,  set 
$
F^{[\omega]} := \bigotimes _{i\in \nat} (F^{\omega_i})^{(i)}
$
 in $ \stp_{\degree{\omega}}$.
\end{nota}

\begin{lem}
\label{lem:prime_field_no_twist}
If $\kf = \fp$, for $\omega \in \seq$, the functors $\forget (F^\omega)$ and $\forget (F^{[\omega]})$ are isomorphic and belong to $\f_{|\omega|} \subset \f$. 
\end{lem}

For $r \in \nat$, consider the $r$th symmetric power $S^r$ as an object of $\stp_r$.

\begin{defn}
\ 
\begin{enumerate}
\item 
For $\omega \in \seq^p (r)$, let $m^\omega : S^{[\omega]} \rightarrow S^r$ be the morphism of $\stp_r$  given by the composite 
\[
S^{[\omega]}
= 
\bigotimes _{i \in \nat} (S^{\omega_i})^{(i)}
\rightarrow 
\bigotimes _{i \in \nat} S^{p^i\omega_i}
\rightarrow
S^r,
\]
where the first map is the tensor product of  the $i$-iterated Frobenius $p$th power maps $(S^{\omega_i})^{(i)}
\rightarrow 
S^{p^i\omega_i}$  and the second is the multiplication of symmetric powers. 
\item 
For $d \in \nat$, let $S^r_{\leq d} \subset S^r$ be:
\[
S^r_{\leq d} 
:= 
\sum_{\substack{\omega \in \seq^p (r) \\ |\omega |\leq d}} \mathrm{image} (m^\omega).  
\]
\end{enumerate}
\end{defn}

\begin{prop}
\label{prop:order_ideal_poly}
\cite{MR1918209,MR1444496} 
For $d, r \in \nat$,
\begin{enumerate}
\item there is  an increasing filtration of $S^r$ in $\stp_r$: 
\[
S^r_{\leq 0} \subset
S^r_{\leq 1} \subset S^r_{\leq 2} \subset \ldots \subset S^r_{\leq r-1} \subset S^r_{\leq r} = S^r;
\]
\item 
the subquotients are given by 
\[
S^r_{\leq d}/ S^r _{\leq d-1} \cong \bigoplus _{\omega\in\seq^p_d (r) }
\strunc ^{[\omega]};
\]
\item 
there is a natural isomorphism  $\forget (S^r_{\leq d} ) \cong p_d S^r$, hence 
$ 
p_d S^r /p_{d-1} S^r 
\cong 
\bigoplus _{\omega \in \seq^p_d (r) } 
\forget (\strunc ^{[\omega]})
$ 
in $\f_d$; 
\item 
if $\kf =\fp$, 
\begin{eqnarray}
 \label{eqn:S_grading}
p_d S^r /p_{d-1} S^r 
\cong 
\forget \Big( 
\bigoplus _{\omega \in \seq^p_d (r) } 
\strunc ^{\omega}
\Big),
\end{eqnarray}
where $\bigoplus _{\omega \in \seq^p_d (r) } 
\strunc ^{\omega} =  \bigoplus _{\omega \in \seq^p_d (r) } 
\bigotimes_{i \in \nat}\strunc^{\omega_i}$ 
is an object of $\stp_d$.
\end{enumerate}
\end{prop}

\subsection{Natural transformations between symmetric power functors}
\label{subsect:nat_trans}

Take   $\kf$ to be the prime field $\fp$. 

\begin{nota}
\label{nota:calap}
As in \cite{MR1269607}, let $\cala (p)$ denote the algebra of Steenrod $p$th powers, with gradings divided by $2$ if $p$ is odd. Thus $\cala (p)$ is generated by the reduced powers $P^i$, where $|P^i|= i (p-1)$ with this grading convention; for $p=2$, $P^i$ is the $i$th Steenrod square $Sq^i$. 

Let $\unst (p)$ denote the category of unstable $\cala(p)$-modules (this category can be  {\em defined} as in  Remark \ref{rem:functors_unstable} below).
\end{nota}

\begin{prop}
\label{prop:nat_trans_S}
\cite{MR1269607}
For $\kf = \fp$
\begin{enumerate}
\item 
$
\hom_\f (S^1, S^n) = \left\{ 
\begin{array}{ll}
\kf & n = p^t \\
0 & \mathrm{otherwise,}
\end{array}
\right.
$

\noindent
with generators the iterated Frobenius maps $S^1 \rightarrow S^{p^t}$, $x \mapsto x^{p^t}$, for $t \in \nat$. 
\item 
The underlying bigraded commutative algebra of $\hom_\f (S^* , S^*)$ is the polynomial algebra 
$
\fp [\hom_\f (S^1, S^*)]. 
$ 
\item 
The functor $S^*$ takes values in  $\unst (p)$. 
\end{enumerate}
\end{prop}

\begin{rem}
\label{rem:functors_unstable}
\ 
\begin{enumerate}
\item 
The algebra structure on $\hom_\f (S^* , S^*)$  is provided by Corollary \ref{cor:expo_hom_bicommutative_Hopf_alg}.
\item
Kuhn   proves the stronger result  that the representation category associated to  $\{ S^n |n \in \nat\}$ is precisely the category $\unst (p)$ of unstable modules over $\cala(p)$.
\item 
For $p=2$, the symmetric algebra $S^* (V)$ identifies with the cohomology $H^* (BV^\sharp; \ft)$ of the classifying space of the dual of $V$ and the $\cala(2)$-action identified above is the usual one. Likewise for odd primes, by identifying  $S^* (V)$ with the polynomial part of $H^* (BV^\sharp; \fp)$, after doubling degrees.
\end{enumerate}  
\end{rem}

\begin{defn}
\label{defn:r}
\cite{MR1269607}
Let  $r : \f \rightarrow \unst (p)$  be the functor  
$r(F)^n := \hom _\f (\Gamma^n, F)$, where Steenrod operations act via natural transformations between divided powers (recalling that $\Gamma^n \cong DS^n$). 
\end{defn}

\subsection{Application to the polynomial filtration of $S^*$}
\label{subsect:poly_S_qa}

Throughout this subsection, $\kf = \fp$. 

\begin{prop}
\label{prop:poly_filt_stp}
For $d \in \nat$,  the $\nat$-graded functor $p_d S^* /p_{d-1} S^*$ has the structure of an $\cala (p)$-module and lies in the image of the forgetful functor $\forget : \stp_d \rightarrow \f_d$.   
\end{prop}

\begin{proof}
Proposition \ref{prop:order_ideal_poly} implies that the $\nat$-graded functor $p_d S^* /p_{d-1} S^*$  lies in the image of $\forget$. 

 By the functoriality of $p_d$, a natural transformation $f : S^m \rightarrow S^n$ induces 
$p_d f : p_d S^m \rightarrow p_d S^n$ and this passes to the subquotients of the filtration, giving the $\cala (p)$-module structure. 
The fact that this action arises from $\stp_d$ follows from Proposition \ref{prop:expo_compare_forget}, using the form of the filtration quotients given by Proposition \ref{prop:order_ideal_poly}. 
\end{proof}

To understand the action of $\cala (p)$ upon $p_d S^* /p_{d-1} S^*$, by Proposition \ref{prop:expo_compare_forget} it suffices to consider it as a functor of $\f^\nat$. Using the results of \cite{MR1443197}, this is succinctly encoded using the 
functor $r$ of Definition \ref{defn:r}. 

\begin{nota}
For $0< i \in \nat$, let $\Phi^i \strunc$ denote the $\nat$-graded exponential functor with 
\[
(\Phi ^i \strunc)^{n} = 
\left\{ 
\begin{array}{ll}
\strunc^{\frac{n}{p^i}} & n \equiv 0 \mod p^i \\
0 & \mathrm{otherwise}.
\end{array}
\right.
\]
\end{nota}

\begin{prop}
\label{prop:poly_filt_mult}
The bigraded functor  $(d,n) \mapsto p_d S^n / p_{d-1} S^n \in \f$ is isomorphic to the  bigraded exponential functor:
\[
 \bigotimes _{i \geq 0} (\Phi^i  \strunc)^*,
\]
using the graded tensor product, where the $d-$degree is given by the polynomial degree.

As a  functor to $\nat$-graded unstable modules, this is isomorphic to $V \mapsto r (\strunc^\bullet  \circ (V \otimes -))$ and, via this isomorphism,  the exponential structure on $p_\bullet S^* / p_{\bullet-1} S^*$ corresponds to the exponential structure of $\strunc^\bullet$. 
\end{prop}

\begin{proof}
The result follows by unravelling the definitions and the identifications underlying Proposition \ref{prop:nat_trans_S}, and by appealing to \cite{MR1443197} to identify the underlying graded unstable module. Namely, there is an isomorphism $r( V \otimes -)\cong V \otimes F(1)$ in $\unst (p)$, where $F(1)$ is the free unstable module on a generator of degree one, which is isomorphic to $\hom_\f (S^1 , S^*)$ as a graded vector space (cf. Proposition \ref{prop:nat_trans_S}). Then \cite[Theorem 1.3(3)]{MR1443197} implies that $r(\strunc^* \circ (V \otimes -))$ is isomorphic to $\strunc^* (V \otimes F(1))$  with induced $\cala (p)$-action. Using the fact that $\strunc^*$ is exponential, the result follows. 
\end{proof}

\begin{rem}
An alternative description of the natural  unstable module structure on $p_d S^* / p_{d-1} S^*$ 
 is given by using the functor 
 $
 \tilde{m}_d : \stp_d \rightarrow \unst (p)
 $
 introduced by Nguyen D.H. Hai in \cite{MR2651573}. 
Explicitly, $p_d S^* / p_{d-1} S^*$ is isomorphic to the functor 
$
 V\mapsto 
\tilde{m}_d \big( \strunc^d (V \otimes -)\big). 
$ 
 This also explains why the Steenrod operations act via natural transformations of $\stp$. 
\end{rem}

\section{The functors $Q^*$}
\label{sect:Q}

The  indecomposables $Q^*$ for the action of $\cala(p)$ on $S^*$ are introduced here, together with their subquotients and approximations. Throughout, $\kf = \fp$.

\subsection{Indecomposables and their approximations}

\begin{defn}
\label{defn:Q}
Let $Q^*$ be the $\nat$-graded functor 
$
Q^* := \fp \otimes_{\cala (p)} S^* 
$
for the $\cala(p)$-action of Proposition \ref{prop:nat_trans_S}.
\end{defn}

\begin{rem}
By Remark \ref{rem:functors_unstable} (up to grading) $Q^*(V)$ is the space of indecomposables for the action of $\cala (p)$ on the polynomial part of $H^* (B V^\sharp; \fp)$. For fixed $V$, the study of $Q^*(V)$ is often known as the Peterson hit problem (of rank $\dim V$) - see \cite{MR3729477,MR3729478} and the references therein for the $2$-primary case.
\end{rem}

\begin{lem}
\label{lem:Q_elementary}
For $n \in \nat$, $Q^n$ is a  finite functor with cosocle $ \strunc^n$ and has polynomial degree $n$.  
\end{lem}

\begin{defn}
\label{defn:induced_filt_Q_Qa}
For $d, n \in \nat$, let 
\begin{enumerate}
\item
$Q^n [d] \subset Q^n $ be the image of $p_d S^n$ in $Q^n$; 
\item 
$Q^n_d$ be the subquotient $Q^n [d]/Q^n [d-1]$;
\item 
$\qa^n_d$ be the degree $n$ part of $\fp \otimes_{\cala (p)} (p_d S^*/p_{d-1}S^*)$.
\end{enumerate}
\end{defn}

For small $n-d$, the functors $\qa^n_d$ are familiar: 
 
\begin{exam}
\label{exam:qa_small}
Let $p=2$ and $n \in \nat$, then  
\begin{enumerate}
\item
$\qa^n_n \cong \Lambda^n$; 
\item
$\qa^n_{n-1} \cong D W_{2, 1^{n-3}}$, in particular is zero for $n \leq 2$. 
\end{enumerate}
As in Proposition \ref{prop:qa_versus_Q} below, these coincide with the respective $Q^n_d$.
\end{exam}

The following is straightforward:

\begin{lem}
\label{lem:Q_poly_filtration}
For $d, n \in \nat$, 
\begin{enumerate}
\item 
$Q^n_d= \qa^n_d=0$ if $d>n$; 
\item 
there is a natural surjection 
$
\qa^n_d 
\twoheadrightarrow 
Q^n_d.
$
\end{enumerate}
\end{lem}

Moreover, Proposition \ref{prop:poly_filt_stp} implies the following:

\begin{cor}
\label{cor:stp_pass_indec}
Let $d \in \nat$.  The components $\qa_d^*$ of the $\nat$-graded functor 
$\fp \otimes _{\cala (p)} (p_d S^* /p_{d-1} S^*)$ lie in the image of the forgetful functor $\forget : \stp_d \rightarrow \f$.  Moreover, for $n \in\nat$, there is a natural surjection in $\stp_d$
\[
\bigoplus _{\omega \in \seq^p_d (n) } 
\bigotimes_{i \in \nat}\strunc^{\omega_i}
\twoheadrightarrow 
\qa^n_d.
\]
\end{cor}

\begin{rem}
As above in Corollary \ref{cor:stp_pass_indec}, $\qa^n_d$ can be considered either as an object of $\stp_d$ or, via the forgetful functor $\forget : \stp_d \rightarrow \f_d \subset \f$, as an object of $\f$. 
\end{rem}

\begin{rem}
\ 
\begin{enumerate}
\item
The fact that $\qa^n_d$ lies in $\stp_d$ allows the highest weight structure of the category of strict polynomial functors to be exploited.
\item
The canonical morphism $\qa^n_d \rightarrow Q^n_d$ is not always an isomorphism (see Corollary \ref{cor:qa_different_Q}). This stems from the fact that $\fp \otimes_{\cala (p)} - $ is right exact but is not exact. Thus $Q^n_d$ need not arise from $\stp_d$.
\end{enumerate}
 \end{rem}

\subsection{An isomorphism criterion}
 
\begin{nota}
For $d, n \in \nat$, let  $K^n_d$ be the kernel of the natural surjection $
\qa^n_d 
\twoheadrightarrow 
Q^n_d
$.
\end{nota}

The aim of this section is to give a criterion for the vanishing of $K^n_d$ or, equivalently, for $Q^n_d$ to be isomorphic to $\qa^n_d$.

For $F \in \ob  \f$ finite, its class in the Grothendieck group $G_0(\f)$ is denoted $[F]$. Moreover,  $G_0 (\f)$ is equipped with the usual lattice structure induced by  $(\zed, \leq)$, which gives the operation $\wedge$ used below:

\begin{prop}
\label{prop:control}
For $\leq n \in \nat$, there is an inequality   in $G_0 (\f)$:
\[
[K^n_d] \leq 
[\qa^n_d] \wedge  
\big[
\bigoplus_{i=0}^{[\log_p (n/p)]}
S^{n-p^i(p-1)} /p_d S^{n-p^i(p-1)}
\big].
\]
\end{prop}

\begin{proof}
Since $\cala (p)$ is generated by the operations $P^{p^i}$  and $|P^{p^i}|= p^i (p-1)$,  it suffices to consider the image in $\qa^n_d$ of these operations. The instability condition from $\unst (p)$ gives the upper bound  $[\log_p (n/p)]$. 
\end{proof}

The following is a consequence of the definition of $\qa^*_d$ as $\cala (p)$-indecomposables: 

\begin{lem}
\label{lem:inequality_QA}
For $d, n \in \nat$, there is an inequality in $G_0 (\f)$:
\[
\big[ p_d S^n /p_{d-1} S^{n} \big]
\leq 
\sum_{m\leq n} \dim_{\fp} (\cala (p)^{n-m} )[\qa^m _d ].
\]
\end{lem}

Proposition \ref{prop:control} thus yields:

\begin{cor}
\label{cor:refine_first_occurrence}
For $\leq n \in \nat$, the surjection  $\qa^n_d \rightarrow Q^n_d$ is an isomorphism if 
\[
[\qa^n_d] \wedge 
 \big[
\qa^{m}_e
\big]
=0
\]
for all $0 \leq m  <n$ and $e>d$. In particular, this holds  if the composition factors of each $\qa^{m}_e$   are all indexed by $p$-restricted partitions.
\end{cor}

As an example application, one has:

\begin{prop}
\label{prop:qa_versus_Q}
Let $p=2$ and suppose that $d> 2 (n-d)$. Then 
 $ 
\qa^n_d \twoheadrightarrow Q^n_d
 $
is an isomorphism if:
\begin{enumerate}
\item 
$n-d \leq 5$; 
\item 
$n-d = 6 $ and $n \leq 18$. 
\end{enumerate}
\end{prop}

\begin{proof}
The strict inequality  $d>2( n-d)$ serves to eliminate the exceptional case $\lambda= (d, d)$ for $n=3d$.  Direct calculation shows that the composition factors of $\qa^n_d$ are all $2$-restricted if $n-d \leq 3$ whereas, for $n-d = 4$ and $n \geq 8$, the sequence $(n-6,1,1)$ means that $\qa^n_{n-4}$ contains a composition factor $L_{(3, 1^{n-7})}$; here $(3, 1^{n-7})$ is not $2$-restricted.

Thus the first statement follows from Corollary \ref{cor:refine_first_occurrence}. Similarly, the second statement follows by an elementary analysis using the hypothesis $d>2( n-d)$.
\end{proof}

\begin{rem}
The first potential failure  of  $\qa^{d+6}_d \twoheadrightarrow Q^{d+6}_d$ to be an isomorphism is for $d=13$,  due to the presence of the composition factor $L_{(3, 1^{11})}$ in $\qa^{18}_{14}$. Now 
$\forget (L_{(3, 1^{11})}) \cong L_{(2,1^{11})} \oplus L_{(1^{13})}$ in $\f$, hence the injectivity  of the above map may only be verified after evaluation on $V=\ft^{\oplus 13}$ by Proposition \ref{prop:simple_stp}; non-injectivity could possibly be detected on  $V=\ft^{\oplus 12}$.
\end{rem}

\subsection{Low-dimensional calculations over $\ft$}

This subsection reviews the low-dimensional behaviour of the functors $\qa^n_d \in \ob \stp_d$ to illustrate the theory.

\begin{prop}
\label{prop:qa_n8}
Let $\kf = \ft$. The non-zero functors $\qa^n_d $ for $n \leq 8$ are:
\[
\begin{array}{|l||l|l|l|l|l|l|l|l|}
\hline
d &\multicolumn{8}{|c|}{ \qa^n_d } \\
\hline
8& &&&&&&&\Lambda^8 
\\
\hline
7& &&&&&&\Lambda^7 & L_{(2,1^5)} 
\\
\hline 
6& &&&&&\Lambda^6& L_{(2,1^4)} \cdot \Lambda^6 & L_{(2^2,1^2)} \cdot \Lambda^6 \cdot L_{(2,1^4)} 
\\
\hline 
5& &&&&\Lambda^5 & L_{(2,1^3) } & L_{(2^2,1)} \cdot \Lambda^5 & L_{(2,1^3)}
\\
\hline
4& &&&\Lambda^4 & L_{(2,1^2)} \cdot \Lambda^4 & \sg{L_{(2^2)}} \cdot L_{(2,1^2)} & \Lambda^4 & \sg{\qa^8_4}\\
\hline
3& && \Lambda^3 & L_{(2,1)} & & & \sg{L_{(3)}} \cdot \Lambda^3  & \\
\hline 
2& & \Lambda^2 & \sg{L_{(2)}} \cdot \Lambda^2 & &&&&
\\
\hline
1& \Lambda^1 &&&&&&&
\\
\hline \hline 
n& 1&2&3&4&5&6&7&8
\\
\hline
\end{array}
\]
in which the terms containing composition factors indexed by partitions that are not $2$-restricted are boxed and $X \cdot Y$ represents an object occurring in an extension $0 \rightarrow X \rightarrow \mathscr{E}\rightarrow Y \rightarrow 0$. 

Moreover, in the Grothendieck group $G_0 (\stp_4)$: 
\[
[\qa^8_4] = [L_{(3,1)}]+ [L_{(2^2)}] +  [ \Lambda^4]  + [ L_{(2,1^2)}].
\]
\end{prop}

\begin{proof}
(Indications.)
The result is proved by standard methods. The functors $\qa^n_n$, $\qa^n_{n-1}$ and $\qa^n _{n-2}$ fall into periodic families that are well understood.  This leaves the cases of  $\qa^7_4$, $\qa^7_3$ and $\qa^8_5$, $\qa^8_4$ which are established by direct calculation. 
\end{proof}

By Proposition \ref{prop:control} and its corollaries, it is  the boxed terms  that play a rôle in determining the associated functors $Q^n_d$. 

\begin{cor}
\label{cor:qa_different_Q}
Let $\kf = \ft$. For $n \leq 8$ the surjection
 $
\qa^n_d \twoheadrightarrow Q^n _d
$ 
is an isomorphism except in the case $(n,d)= (7,3)$, when  there is a short exact sequence 
\[
0
\rightarrow 
\Lambda^2 
\rightarrow 
\qa^7_3 
\rightarrow 
Q^7_3 
\rightarrow 
0,
\]
where the kernel corresponds to the composition factor $L_{(2^2)} $ of $\qa^6_4$.
\end{cor}

\begin{proof}
Proposition \ref{prop:control} together with inspection of the result of Proposition \ref{prop:qa_n8} 
show that the only case for which the surjection is potentially not an isomorphism is $(n,d)= (7,3)$. 

Now, $\qa^7_3 \cong S^3$ which is a uniserial functor with socle series $
\Lambda^2 , \Lambda^1 , \Lambda^2 , \Lambda^3.
$
Hence, to prove the result, it suffices to show that $\Lambda^1$ is in the socle of $Q^7_3$, which is equivalent  to showing that $\Lambda^1$ is in the socle of $Q^7$.  This is proved by a straightforward argument involving the operation $Sq^1$: the composition factor 
$\Lambda^1$ is detected in the socle of $S^8$.
 \end{proof}

\section{Representation stability for strict polynomial functors}
\label{sect:stp_rp}

The purpose of this section is to recall Harman's representation stability result for strict polynomial functors,  stated below as Theorem \ref{thm:representation_stability}.

\subsection{Representation stability}
\label{subsect:rep_stab_Harman}

\begin{nota}
\label{nota:strat_prt}
For $c, d \in \nat$, let $\prt_d ^{\geq c} \subset \prt_d$ denote the subset of partitions $\mu$ such that $\mu_0' \geq c$ (respectively  $\prest^{\geq c}_d \subset \prest_d$  for $p$-restricted partitions). 
\end{nota}

The partial order $\trianglelefteq$ is compatible with the subsets $\prt_d ^{\geq c}$ in the following sense:

\begin{lem}
For $c, d \in \nat$, and partitions  $\lambda, \mu \in  \prt_d$ such that $\mu \trianglelefteq\lambda$,  if $\lambda \in \prt_d ^{\geq c}$ then $\mu \in \prt_d ^{\geq c}$.
\end{lem}

\begin{prop}
\label{prop:stratify_hwc}
 \cite{2015arXiv150906414H}
For $c, d \in \nat$, the category $\stp_d^{\geq c}$ is a highest weight category with weights $(\prt_d^{\geq c}, \trianglelefteq)$ and with standard objects $\{ W_\mu | \mu \in \prt_d^{\geq c}\}$. 
\end{prop}

\begin{rem}
Harman \cite[Section 2.3]{2015arXiv150906414H}  works  with the category of modules over the appropriate Schur algebra. The category $\stp_d^{\geq c}$ here corresponds to $\mathcal{S} (N, d)^{\leq d-c}$ of {\em loc. cit.}, for an integer $N \geq d$.
\end{rem}

When $\kf$ is a field of characteristic $p>0$, Harman \cite{2015arXiv150906414H}  shows that the  categories $\stp_d^{\geq d-t}$ (for fixed  $t$ and varying $d$) exhibit a form of representation stability. First observe the following combinatorial stability lemma:

\begin{lem}
\label{lem:comb_stab}
For integers $t \leq d \leq e \in \nat$, the map of sets
\begin{eqnarray*}
\prt^{\geq d-t}_d 
&
\rightarrow &
\prt_e^{\geq e-t}\\
\lambda 
&\mapsto & \lambda \bullet 1^{e-d}
\end{eqnarray*}
given by concatenation of partitions is an injection and is a bijection if $d \geq 2t$.  For a prime $p$, this restricts to $ \prest^{\geq d-t}_d \rightarrow \prest^{\geq e-t}_d$ that is a bijection if  $d \geq 2t$.
\end{lem}

It is convenient to introduce the following terminology:

\begin{defn}
\label{defn:stable_pair}
A pair of natural numbers $(d,t) \in \nat^{\times 2}$ is {\em stable} if $d\geq  2t$ and {\em strictly stable} if $d > 2t$.
\end{defn}

The following is adopted throughout the text:

\begin{conv}
$\lceil \log_p 0 \rceil$ is taken to be zero. 
\end{conv}

\begin{thm}
\label{thm:representation_stability}
\cite[Theorem 2.8]{2015arXiv150906414H}
Let $\kf$ be a field of characteristic $p>0$, $(d,t)\in \nat ^{\times 2}$ be a stable pair and $e> d$ an integer such that $d \equiv  e \mod p^{\lceil \log_p t \rceil}$.

The categories $\stp_d ^{\geq d-t}$
and $\stp_e ^{\geq e-t}$ are equivalent as highest weight categories with respect to the bijection of 
weights $\prt_d ^{\geq d-t} \cong \prt_e ^{\geq e-t}$ of Lemma \ref{lem:comb_stab}. In particular, the simple object $L_\lambda \in \ob \stp_d^{\geq d-t}$ is sent under this equivalence to 
$L_{\lambda \bullet 1^{e-d}} \in \ob \stp_e^{\geq e-t}$. 
  \end{thm}  

\begin{rem}
\ 
\begin{enumerate}
\item 
The proof  uses Ringel duality for highest weight categories (see \cite[Theorem 6]{MR1128706}): it suffices to show that the endomorphism rings of the  respective characteristic objects (cf. Remark \ref{rem:characteristic}) are isomorphic. By Proposition \ref{prop:characteristic}, this corresponds to establishing the isomorphism of rings:
\[
\mathrm{End}\Big(\bigoplus_{\lambda \in \prt_d^{\geq d-t}} \Lambda^{\lambda'} \Big) 
\cong 
\mathrm{End}\Big(\bigoplus_{\mu \in \prt_e^{\geq e-t}}  \Lambda^{\mu'} \Big). 
\]
The hypothesis upon $d-e$ ensures that this is the case, as can be checked by direct calculation, using the exponentiality of exterior powers.
\item 
\cite[Theorem 2.8]{2015arXiv150906414H} uses the hypothesis $d > 2t$; this can be weakened as above.
\end{enumerate}
\end{rem}

\begin{nota}
\label{nota:per_equiv}
For $d,e,t$ satisfying the hypotheses of Theorem \ref{thm:representation_stability}, let 
\[
\per_{d,e} : \stp_d ^{\geq d-t} \stackrel{\cong}{\rightarrow} \stp_e^{\geq e-t}
\]
denote the equivalence of highest weight categories provided by Theorem \ref{thm:representation_stability}.
\end{nota}

\section{Weak representation stability for polynomial functors}
\label{sect:weak}

Representation stability as in Theorem \ref{thm:representation_stability} is not currently known to hold for the analogous categories of polynomial functors. This section proves a substitute, Theorem \ref{thm:rep_stab_F}, that holds after passage to Grothendieck groups. 

The difficulty stems from the fact that the periodicity equivalences $\per_{d,e}$ have to be replaced by the group morphism induced by concatenation of partitions, together with the fact that the functor $\forget L_\lambda$ is not in general simple if $\lambda$ is not $p$-restricted. The workaround is to use the Steinberg tensor product theorem in conjunction with an induction on the  weight; this relies on understanding the periodicity equivalences on certain tensor products.

\subsection{Tensor products and representation stability}

The tensor product is compatible with representation stability by the following result:

\begin{prop}
\label{prop:tensor_compatible_repstab}
Let $\kf$ be a field of characteristic $p>0$ and $d, e, t  \in \nat$ such that  $(d,t)$ is stable and  $d \equiv e \mod p^{\lceil \log_p t \rceil}$. Consider $d', d'' \in \nat$ such that $d'+d'' =d$ with $d'' \leq t$ and set $t' := t- d''$ and $e'=e-d''$. 

Then the pair $(d',t')$ is stable and the  periodicity equivalences of Theorem \ref{thm:representation_stability} induce exact functors given by the two composites in the diagram: 
\[
\xymatrix{
\stp_{d'}^{\geq d' -t'} \times \stp_{d''} 
\ar[r]^(.6){\otimes }
\ar[d]^\cong_{\per_{d',e'}\times \mathrm{Id}}
&
\stp_d ^{\geq d-t} 
\ar[d]_\cong^{\per_{d,e}}
\\
\stp_{e'}^{\geq e' -t'} \times \stp_{d''}
\ar[r]_(.6){\otimes} 
&
\stp _e^{\geq e-t}.
}
\]
These induce the same morphism on Grothendieck groups, 
$
G_0 ( \stp_{d'}^{\geq d' -t'} \times \stp_{d''} )
\rightarrow 
G_0 (\stp _e^{\geq e-t}),
$
i.e., for all $X \in \ob \stp_{d'}^{\geq d' -t'}$ and $Y \in \ob \stp_{d''}$, 
\begin{eqnarray}
\label{eqn:per_tensor}
[(\per_{d',e'} X) \otimes Y] = [\per_{d,e} (X \otimes Y)] 
\end{eqnarray}
in $G_0 (\stp _e^{\geq e-t})$.
\end{prop} 

The following basic Lemma relates stability for the pairs $(d',t')$ and $(d,t)$:

\begin{lem}
\label{lem:stable_pairs}
For  $d,t \in \nat$ and natural numbers $d', d''$ such that  $d= d' + d''$ and $d'' \leq t$, set $t' := t - d''$. Then $t' \in \nat$ and the following conditions are equivalent:
\begin{eqnarray}
d &\geq & 2t \\
d' & \geq & 2 t ' + d'',
\end{eqnarray}
i.e., $(d,t)$ is stable if and only if $(d',t')$ is stable and $d'' \leq d' - 2t'$.
\end{lem}

\begin{proof}[Proof of Proposition \ref{prop:tensor_compatible_repstab}]
As in Lemma \ref{lem:stable_pairs}, $t \geq t ' \geq 0$ and $(d',t') $ is a stable pair. 
Theorem \ref{thm:representation_stability} therefore provides the periodicity equivalences:
 $ \per_{d',e'} : \stp_{d'}^{\geq d' -t'} 
\stackrel{\cong}{\rightarrow} 
\stp_{e'}^{\geq e'-t'}$ and  
$ \stp_{d}^{\geq d -t} 
\stackrel{\cong}{\rightarrow} 
\stp_{e}^{\geq e-t}$ and hence the vertical functors of the diagram. Moreover, all the functors are exact, hence the two composites induce morphisms between the respective Grothendieck groups. It remains to show that these coincide.

The highest weight structure implies that  $G_0 (\stp_{d'}^{\geq d' -t'} )$ is the free abelian group generated by the classes $[\Lambda^{\lambda'}]$ for $\lambda \in \prt_{d'} ^{\geq d'-t'}$, likewise $G_0 (\stp_{d''})$ is generated 
 by $[\Lambda^{\mu'}]$ for $\mu \in \prt_{d''}$. Thus $G_0 ( \stp_{d'}^{\geq d' -t'} \times \stp_{d''} )$ is generated by the classes $[(\Lambda^{\lambda'},\Lambda^{\mu'})]$. Hence, to prove the result, it suffices to establish the equality (\ref{eqn:per_tensor}) for $X= \Lambda^{\lambda'}$ and $Y = \Lambda^{\mu'}$, with $\lambda$ and $\mu$ is above.

This is a consequence of the explicit nature of the equivalences $\per_{d',e'}$ and $\per_{d,e}$ given by Theorem \ref{thm:representation_stability} together with the rôle played by tensor products of exterior powers in representation stability (cf.  Proposition \ref{prop:characteristic}). The result follows from the stronger property established below:
\begin{eqnarray}
\label{eqn:iso_per_Lambda}
\per_{d,e} ( \Lambda^{\lambda'} \otimes \Lambda^{\mu'}) \cong (\per_{d',e'} \Lambda^{\lambda'}) \otimes \Lambda ^{\mu'}.
\end{eqnarray}
Now, $\per_{d',e'} \Lambda^{\lambda'} = \Lambda^{(\lambda \bullet 1^{e-d})'}$ and, using the similar behaviour of $\per_{d,e}$, to establish (\ref{eqn:iso_per_Lambda}) it suffices to check that $\lambda_0'$ is the supremum of the sequences $\lambda'$ and $\mu'$.  This follows from the hypothesis  $\lambda'_0 \geq d' -t'$, the identification  $|\mu' | = d''$ and the fact that $d'' \leq d' -2t' \leq d' - t '$, by Lemma \ref{lem:stable_pairs}.
\end{proof}

\subsection{The Steinberg tensor product theorem and representation stability}

It is useful to understand the behaviour of the simple objects under the equivalence of Theorem \ref{thm:representation_stability} over a field of characteristic $p>0$. The Steinberg tensor product theorem,  Theorem \ref{thm:Steinberg_tensor_prod} provides the isomorphism 
\[
L_\lambda 
\cong 
L_{\lambda[0]} 
\otimes 
L_{\overline{\lambda}}^{(1)}
\]
for a simple object, with $\lambda[0]$ a $p$-restricted partition, which is used in the following result.

\begin{prop}
\label{prop:simples_steinberg_per}
For $\kf$ a field of characteristic $p>0$, let $L_\lambda \in \ob \stp_d^{\geq d-t}$ be a simple object, where $(d,t)$ is a stable pair. Set $d'= |\lambda [0]|$ and $d'' =p |\overline{\lambda}|$ (so that $d'+d'' =d$) and $t':= t-d''$.  

\begin{enumerate}
\item 
If $\lambda [0]= 0$, then $p=2$, $d=2t$ and $L_\lambda \cong (\Lambda^t) ^{(1)}$. 
\item 
Otherwise $(d',t')$ is stable and $L_{\lambda [0]} \in \ob \stp_{d'} ^{\geq d' -t'}$. 
Moreover, writing $\tilde{d} := d' + |\lambda |$ and $\tilde{t}:= t' + |\lambda |$, the pair $(\tilde{d}, \tilde{t})$ is stable and $L_{\lambda [0]} \otimes L_{\overline{\lambda}} \in \ob \stp_{\tilde{d}} ^{\geq \tilde{d} -\tilde{t}}$.
\end{enumerate}
\end{prop}

\begin{proof} By Proposition \ref{prop:tensor_compatib}, to show that $L_{\lambda [0]}$ lies in  $\stp_{d'} ^{\geq d' -t'}$, it suffices to show that $|\overline{\lambda}| < d-t$ apart from in the exceptional case when $\lambda [0]=0$. 

Suppose that $|\overline{\lambda}| \geq  d-t$, then the inequality $p |\overline{\lambda} | \leq d$ implies that $p (d-t) \leq d$ so that $(p-1)d \leq pt$; this is a strict inequality unless both $\lambda[0] =0$ and $|\overline{\lambda}| = d-t$. 

The stability hypothesis gives $d \geq 2t$, so that the inequality gives $2(p-1) \leq pt$. This is a contradiction unless $p=2$ and both $\lambda[0] =0$ and $|\overline{\lambda}| = d-t$. In this remaining case, it is straightforward to check that one must have $d=2t$ and $L_\lambda \cong (\Lambda^t) ^{(1)}$.

In the case $\lambda [0]\neq 0$, $L_{\lambda[0]}$ is a non-zero object of $\stp_{d'}^{\geq d' -t'}$, in particular this implies $t'\geq 0$. Hence Lemma \ref{lem:comb_stab} implies that $(d',t')$ is stable. Similarly,  $(\tilde{d}, \tilde{t})$ is stable. 
\end{proof}

  \begin{cor}
  \label{cor:simples_periodicity}
  Let $\kf$ be a field of characteristic $p>0$ and $d, e, t  \in \nat$ such that  $(d,t)$ is stable and  $d \equiv e \mod p^{\lceil \log_p t \rceil}$.
   For $L_\lambda 
\cong 
L_{\lambda[0]} 
\otimes 
L_{\overline{\lambda}}^{(1)}$ a simple object of $\stp_d^{\geq d-t}$ with $\lambda [0] \neq 0$, 
\begin{eqnarray*}
\per_{d,e} L_\lambda 
&\cong& 
(\per_{d',e'} L _{\lambda [0]}) 
\otimes 
 L_{\overline{\lambda}}^{(1)},
 \end{eqnarray*}
 where $d' = |\lambda [0]|$ and $e':= d' + (e-d)$.
 
Moreover, setting 
 $\tilde{d}:= d'+ |\overline{\lambda}|$ and $\tilde{e}:= \tilde{d} + (e-d)$, 
 there is an equality in the Grothendieck group $G_0 (\stp _{\tilde{e}})$:
\begin{eqnarray*} 
 [\per_{\tilde{d}, \tilde{e}} (L _{\lambda[0]} \otimes L_{\overline{\lambda}})]
 &= &
 [(\per_{d',e'} L _{\lambda [0]}) 
\otimes 
 L_{\overline{\lambda}}].
\end{eqnarray*}
\end{cor}
  
\begin{proof}
In the notation of Proposition \ref{prop:simples_steinberg_per}, the hypotheses imply that $(d,t)$, $(d', t')$ and $(\tilde{d}, \tilde{t})$ are stable pairs. In particular, the periodicity equivalences $\per_{d,e}$, $\per_{d',e'}$ and $\per_{\tilde{d}, \tilde{e}}$ are defined.  The identification of $\per_{d,e}L_\lambda$ then follows by a straightforward analysis of the partition $\lambda \bullet 1^{e-d}$.

The object $L_{\lambda[0]} \otimes L_{\overline{\lambda}}$ is in general not  simple, so the previous argument does not apply in this case. Here  Proposition \ref{prop:simples_steinberg_per} together with Proposition \ref{prop:tensor_compatible_repstab} provide the stated equality in the Grothendieck group.
\end{proof}

\subsection{Comparison via $\forget$}

In this subsection, $\kf$ is the prime field $\fp$. Recall from Proposition \ref{prop:classify_simples_fp} that $G_0 (\f)$ is the free abelian group on the set $\prest$ of $p$-restricted partitions.  

\begin{nota}
For $d<e \in \nat$, let $\bullet 1^{e-d} : G_0 (\f) \rightarrow G_0 (\f)$ denote the morphism of abelian groups induced by the   concatenation of partitions (cf. Lemma \ref{lem:comb_stab}).
\end{nota}

\begin{lem}
\label{lem:restrict_concat}
For $d<e \in \nat$ and $c \in \nat$, $\bullet 1^{e-d} : G_0 (\f) \rightarrow G_0 (\f)$ restricts to:
\begin{eqnarray*}
\bullet 1^{e-d} &:& G_0 (\f_d) \rightarrow G_0 (\f_e)\\ 
\bullet 1^{e-d} &:& G_0 (\f_d^{\geq c}) \rightarrow G_0 (\f_e^{\geq e-d+c}).
\end{eqnarray*}
\end{lem}

For  $c,d \in \nat$, the Grothenieck group  $G_0 (\stp_d^{\geq c})$  is the free abelian group on $\prt_d^{\geq c}$; similarly,  $G_0 (\f_d^{\geq c})$ is the free abelian group on $\prest_d ^{\geq c}$. Proposition \ref{prop:classify_simples_fp} implies:

\begin{lem}
\label{lem:forget_G0}
For $c, d \in \nat$, the functor $\forget : \stp_d^{\geq c} \rightarrow \f_d ^{\geq c}$ induces a surjective morphism of abelian groups 
$
G_0 (\forget) : 
G_0 (\stp_d^{\geq c})
\twoheadrightarrow 
G_0 (\f_d^{\geq c}).
$ 
\end{lem}

\begin{thm}
\label{thm:rep_stab_F}
Let $\kf = \fp$  and suppose that  $(d,t)$ is a stable pair of integers (strictly stable if $p=2$).  For  $d<e \in \nat$ such that $d \equiv e \mod p^{\lceil \log_p t \rceil}$,  the periodicity equivalence $\per_{d,e} : \stp_d^{\geq d-t} \rightarrow \stp_e^{\geq e-t}$ induces a commutative diagram of abelian groups:
\[
\xymatrix{
G_0 (\stp_d^{\geq d -t})
\ar[d]_{G_0(\per_{d,e})}^\cong
\ar[r]^{G_0 (\forget)} 
&
G_0 (\f_d^{\geq d -t})
\ar[d]^{\bullet 1^{e-d}}_\cong
\\
G_0 (\stp_e^{\geq e-t})
\ar[r]_{G_0 (\forget)} 
&
G_0 (\f_e^{\geq e -t})
}
\]
in which the vertical morphisms are isomorphisms.
\end{thm}

\begin{proof}
By Lemma \ref{lem:restrict_concat}, it suffices to prove commutativity after composing with the canonical inclusion of abelian groups $G_0 (\f_e ^{\geq e-t})\subset G_0 (\f)$. The proof is by induction upon $d$. For $0 \leq d <p$ the result holds by inspection, since all partitions $\lambda$ with $|\lambda|<p$ are $p$-restricted.

For the inductive step,  it suffices to check commutativity on the generators $\{ [L_\lambda] | \lambda \in  \prt_d^{\geq d-t}\}$ of $G_0 (\stp_d^{\geq d -t})$. For  $p$-restricted partitions $\lambda$, commutativity is clear, using the behaviour of $\per_{d,e}$ on the simple objects given in Theorem \ref{thm:representation_stability}.

If $\lambda$ is not $p$-restricted, write $L_\lambda \cong L_{\lambda [0]} \otimes L_{ \overline{\lambda}}^{(1)}$ using the Steinberg tensor product theorem (Theorem \ref{thm:Steinberg_tensor_prod}), where $\overline{\lambda} \neq 0$, since $\lambda$ is not $p$-restricted. Working over $\kf= \fp$, there is an isomorphism $\forget L_\lambda \cong \forget (L_{\lambda [0]} \otimes L_{ \overline{\lambda}})$, where $L_{\lambda [0]} \otimes L_{ \overline{\lambda}}\in  \ob \stp_{\tilde{d}}^{\geq \tilde{d}- \tilde{t}}$ with $(\tilde{d}, \tilde{t})$ stable, by Proposition \ref{prop:simples_steinberg_per}. Since $\overline{\lambda} \neq 0$, $\tilde{d} < d$, hence the inductive hypothesis applies in this case. 

To conclude, by the inductive hypothesis, it suffices to show that 
\[
[\forget \big(\per_{\tilde{d}, \tilde{e}} (L_{\lambda [0]} \otimes L_{ \overline{\lambda}} )\big)] 
= 
[\forget \big(\per_{d, e} (L_{\lambda [0]} \otimes L_{ \overline{\lambda}}^{(1)} )\big)] 
\]
in $G_ 0 (\f)$. This follows from Corollary \ref{cor:simples_periodicity}.
\end{proof}

\section{Representation stability for $Q^*$}
\label{sect:rep_stab}

In this Section, the representation stability results Theorem \ref{thm:representation_stability} (for strict polynomial functors) and Theorem \ref{thm:rep_stab_F} (for polynomial functors) are applied to the study of the functors $Q^n$. 

For this the prime $p$ has to be taken to be $2$. This ensures, by Proposition \ref{prop:order_ideal_poly}, that the functors appearing as subquotients of the polynomial filtration of $S^n$ are direct sums of tensor products of exterior power functors, in particular are {\em tilting objects} for the highest weight theory (see Remark \ref{rem:characteristic}). Whilst odd primary analogues of the results are expected to hold, they do not follow immediately from the theory developed here.

\subsection{Representation stability at the prime $2$}

\begin{lem}
\label{lem:connectivity_pdSn}
Let $\kf = \ft$. For $a, d, n \in \nat$ with $d \leq n$. 
\begin{eqnarray*}
\qa^n_d,\  p_d S^n /p_{d-1}S^n &\in &\ob  \stp_d ^{\geq d - (n-d)}
\\
Q^n/ Q^n[d-1], \ S^n / p_{d-1} S^n
&\in& 
\ob \f^{\geq d -  (n-d)} _ n.
\end{eqnarray*}
\end{lem}

\begin{proof}
It  suffices to show that $ p_d S^n /p_{d-1}S^n$ is in the image under $\forget$ of  $\stp_d ^{\geq d - (n-d)}$. Hence, by Corollary \ref{cor:stp_pass_indec}, it suffices to show that, for $\omega \in \seq^p_d (n)$, the functor   
$\bigotimes_{i \in \nat}\strunc^{\omega_i}$ lies in $\stp_d^{\geq d - (n-d)}$. Here, since  $p=2$, $\strunc^{\omega_i}$ identifies with the exterior power functor $\Lambda^{\omega_i}$.

By hypothesis 
$|\omega | = \omega _0 + |\omega^-| = d$, where $\omega^-$ is the sequence obtained by deleting $\omega_0$. Similarly, $\degt{\omega} = \omega_0 + 2 \degt{\omega^-} = n$. Now $\degt{\omega^-} \geq |\omega^- | = d- \omega_0$, hence $n \geq \omega_0 + 2 (d- \omega_0)$, which gives $\omega_0 \geq d - (n-d)$. This implies the result.
\end{proof}

\begin{thm}
\label{thm:qa_rep_stab}
Let $\kf = \ft$. 
For natural numbers $d \leq n$ such that $(d, n-d) $ is stable, if $e > d \in \nat$ such that $e \equiv d \mod 2^{\lceil \log_2 (n-d) \rceil}$,  under the equivalence of categories   $\per_{d,e} 
 :
\stp_d^{\geq d- (n-d)}
\stackrel{\cong}{\longrightarrow}
 \stp_{e}^{\geq e - (n-d)} $ 
 of Theorem \ref{thm:representation_stability}, $\qa^n_d$ is sent to 
$
\qa_{e}^{n + e - d }$.
\end{thm}

\begin{proof}
This is a straightforward application of Theorem \ref{thm:representation_stability}. Recall from Corollary \ref{cor:stp_pass_indec} that the diagram giving rise to $\qa^n _d $ as a colimit lies entirely within the category  $\stp_d^{\geq d- (n-d)}$ and the objects of the diagram are direct sums of tensor products of exterior power functors.

The equivalence of categories given by Theorem \ref{thm:representation_stability} is established using the structure of the morphisms between such objects. It follows that, under the equivalence of categories $\per_{d,e}$, the diagram defining $\qa^n_d$ is sent  to the respective diagram for $ \qa_{e}^{n + e - d }$.
\end{proof}

Theorem \ref{thm:qa_rep_stab} together with Theorem \ref{thm:rep_stab_F} imply:

\begin{cor}
\label{cor:qa_rep_stab_F}
Let $\kf = \ft$,   $d, e, n  \in \nat$ satisfy the hypotheses of Theorem \ref{thm:qa_rep_stab} and, in addition, suppose that the pair $(d, n-d)$ is  strictly stable. Then, under the isomorphism
$
\bullet 1^{e-d} : 
G_0 (\f_d^{\geq d - (n-d)})
\stackrel{\cong}{\rightarrow} 
G_0 (\f_d^{\geq e - (n-d)})$ of Grothendieck groups, $ [\qa^n_d]$ maps to 
$
[\qa_e^{n+e-d}]$.
\end{cor}

\begin{rem}
\label{rem:rep_stab_F}
Corollary \ref{cor:qa_rep_stab_F} has to be stated for Grothendieck groups, since the categories 
$\f_d^{\geq d - (n-d)}$ and $\f_d^{\geq e - (n-d)}$ are not currently known to be equivalent, whereas their Grothendieck groups are isomorphic.
\end{rem}

\subsection{Conjectural periodicity for $Q^*/Q^* [d-1]$}

For $0< d\leq n \in \nat$, Lemma \ref{lem:connectivity_pdSn} shows that $Q^n /Q^n [d-1]$ lies in $\f_n ^{\geq d -  (n-d)}$. Similarly, if $e>d$, $Q^{n+ e-d} /Q^{n+ e-d} [e-1]$ lies in $\f_{n+ e-d} ^{\geq  e   -  (n-d)}$. Corollary \ref{cor:qa_rep_stab_F} suggests the following:

\begin{conj}
\label{conj:rep_stab_Q}
Suppose that $d, e, n \in \nat$ satisfy the hypotheses of Corollary \ref{cor:qa_rep_stab_F},
then under the isomorphism $\bullet 1^{e-d} : 
G_0 (\f_n^{\geq d - (n-d)})
\stackrel{\cong}{\rightarrow} 
G_0 (\f_{n+e-d}^{\geq e - (n-d)})$
 of Grothendieck groups, 
$  \big[Q^n/ Q^n[d-1] \big]$ maps to
$\big[Q^{n+ e-d} /Q^{n+ e-d} [e-1]\big ]$.
\end{conj}

\begin{rem}
Corollary \ref{cor:qa_rep_stab_F} implies that Conjecture \ref{conj:rep_stab_Q} holds if $Q^n_i \cong \qa^n _i$ for all $i>d$ and $Q^{n+ e-d}_j \cong \qa^{n+ e -d}_j$ for all $e>d$. Corollary \ref{cor:qa_different_Q} shows that $\qa^7_ 3 \not \cong Q^7_3$; however, this does not satisfy the stability hypothesis. 
\end{rem}

In particular, Proposition \ref{prop:qa_versus_Q} implies the following (corresponding to a zone in which the representation theory is well understood):

\begin{prop}
Conjecture \ref{conj:rep_stab_Q} holds if $n-d \leq 5$. 
\end{prop}

It is possible to strengthen the conjecture, bearing in mind Remark \ref{rem:rep_stab_F}, as follows:

\begin{conj}
\label{conj:rep_stab_Q_strong}
Suppose that $d, e, n \in \nat$ satisfy the hypotheses of Corollary \ref{cor:qa_rep_stab_F}, then 
the lattices of subobjects of $Q^n /Q^n[d-1]$ and $Q^{n+e-d}/ Q^{n+e-d}[e-1]$ are isomorphic, compatibly with the identification of  Conjecture \ref{conj:rep_stab_Q}.
\end{conj}

\begin{rem}
Further structure is available to study the functors $Q^n$ and their quotients. In particular, the coproduct on symmetric powers induces a graded coalgebra structure on $Q^*$, which gives tools for comparing the $Q^n$. For instance, at $p=2$ and for $i, n \in \nat$, the composite $Q^n \rightarrow Q^{n - 2^i } \otimes Q^{2^i} \rightarrow  Q^{n - 2^i } \otimes \Lambda^{2^i}$ given by composing the coproduct with the morphism induced by projection to the cosocle of $Q^{2^i}$ has adjoint $
(Q^n : \Lambda^{2^i}) 
\twoheadrightarrow 
Q^{n-2^i},
$ 
where $( -: \Lambda^{2^i}) $ is the division functor. This can be exploited in studying the above Conjectures.
\end{rem}

\providecommand{\bysame}{\leavevmode\hbox to3em{\hrulefill}\thinspace}
\providecommand{\MR}{\relax\ifhmode\unskip\space\fi MR }
\providecommand{\MRhref}[2]{%
  \href{http://www.ams.org/mathscinet-getitem?mr=#1}{#2}
}
\providecommand{\href}[2]{#2}

\end{document}